\documentclass[11pt,reno]{amsart}
\usepackage{amssymb}
\usepackage{amsmath}
\usepackage{enumitem}
\usepackage{mathrsfs}
\usepackage[english]{babel}
\usepackage[all]{xy}
\usepackage{hyperref}
\usepackage{marginnote}

\usepackage{stmaryrd}

\usepackage[margin=1.05in]{geometry}

\RequirePackage{mathrsfs} 

\newtheorem{theorem}{Theorem}[section]
\newtheorem{lemma}[theorem]{Lemma}
\newtheorem{proposition}[theorem]{Proposition}
\newtheorem{corollary}[theorem]{Corollary}
\newtheorem{conjecture}[theorem]{Conjecture}

\theoremstyle{definition}
\newtheorem*{ack}{Acknowledgements}
\newtheorem*{con}{Conventions}
\newtheorem{remark}[theorem]{Remark}
\newtheorem{example}[theorem]{Example}
\newtheorem{definition}[theorem]{Definition}

\numberwithin{equation}{section} \numberwithin{figure}{section}

 \DeclareMathOperator{\NS}{NS}
\DeclareMathOperator{\Aut}{Aut}

\DeclareMathOperator{\Spec}{Spec}

\DeclareMathOperator{\an}{an}

\DeclareMathOperator{\Hom}{Hom}

 \DeclareMathOperator{\End}{End}

\newcommand{\Qbar}{\overline{\QQ}}

\newcommand\ZZ{\mathbb{Z}}

\newcommand\QQ{\mathbb{Q}}

\newcommand\CC{\mathbb{C}}

\newcommand\OO{\mathcal{O}}

\usepackage{color}

\title[Arithmetic hyperbolicity]{Arithmetic hyperbolicity:  automorphisms and persistence}

\author{Ariyan Javanpeykar}
\address{Ariyan Javanpeykar \\
Institut f\"{u}r Mathematik\\
Johannes Gutenberg-Universit\"{a}t Mainz\\
Staudingerweg 9, 55099 Mainz\\
Germany.}
\email{peykar@uni-mainz.de}



\begin{document}

\begin{abstract}  We  show that if the automorphism group of a projective variety is torsion, then it is finite. Motivated by Lang's conjecture on  rational points of hyperbolic varieties, we use this to prove that a projective variety with only finitely many rational points has only finitely many automorphisms.  Moreover, 
 we investigate to what extent finiteness of $S$-integral points on a variety over a number field persists over finitely generated fields.  To this end, we introduce the class of mildly bounded varieties and prove a general criterion for proving this persistence. 
 \end{abstract}

\maketitle

\thispagestyle{empty}

\section{Introduction}

Let $X$ be a projective variety over a number field $K$. Suppose that $X(L)$ is finite for every number field $L/K$. Motivated by Lang's conjecture on rational points of hyperbolic varieties, we show that $X$ has only finitely many automorphisms (Theorem \ref{thm1}). In our work    with Junyi Xie, we build on this result and prove the stronger statement that a  quasi-projective variety $X$ over $K$ with $X(L)$ finite for every number field $L/K$ has, in fact, only finitely many birational self-maps; see \cite[Theorem~1.3]{JXie}. For a survey of  our work on self-maps of hyperbolic varieties, we refer the reader to \cite[\S 15]{JBook}.
  
We also investigate the Persistence Conjecture (see Conjecture \ref{conj:pers}) which, roughly speaking, says that the finiteness of $S$-integral points of a variety  over $\QQ$ forces the    finiteness   of points over all $\mathbb{Z}$-finitely-generated integral domains of characteristic zero.   We prove this conjecture under the additional assumption that the variety $X$ is ``mildly bounded'' (Definition \ref{defn:mb}); see Theorem \ref{thm:mb} for a precise statement. In the current paper, this result is used to prove the Persistence Conjecture for all Brody hyperbolic projective varieties; see Theorem \ref{thm:geometricity_intro}. 
  
However, Theorem \ref{thm:mb} is applied in four other situations elsewhere. Namely,  we use Theorem \ref{thm:mb} to prove the Persistence Conjecture for a variety $X$    which admits a quasi-finite period map \cite{JLitt}, or  is hyperbolically embeddable \cite{JLevin}, or admits a quasi-finite map to a semi-abelian variety \cite{vBJK}, or effectively parametrizes some polarized varieties with semi-ample canonical bundle \cite{JSZ}.  

 For a  survey of our work on the Persistence Conjecture and its relation to Lang's conjecture, we refer the reader to \cite[\S17]{JBook}.

\subsection{The Green--Griffiths--Lang conjecture}

A variety $X$ over $\CC$ is Brody hyperbolic if every holomorphic map $\CC\to X^{\an}$ is constant, where $X^{\an}$ is the complex-analytic space associated to $X$.  Conjecturally, the property of being Brody hyperbolic is captured by the finiteness of rational points. Let us be more precise.

  Let $k$ be an algebraically closed field of characteristic zero.
Following \cite[\S4]{JLalg}, we say that a finite type separated scheme $X$ over $k$ is \emph{arithmetically hyperbolic over $k$} if, for      all $\ZZ$-finitely generated subrings $A\subset k$ and all finite type separated schemes $\mathcal{X}$ over $A$  with $\mathcal{X}_k\cong X$, the set $\mathcal{X}(A)$ is finite.

The results of this paper are motivated by the conjectures of Green--Griffiths and Lang on hyperbolic varieties.

\begin{conjecture}[Consequence of conjectures of Green--Griffiths and Lang]\label{conj:lang}
Let $X$ be a proper integral variety over $k$. Then the following are equivalent.
\begin{enumerate}
\item  The projective variety $X$ is arithmetically hyperbolic over $k$.
\item Every integral closed subvariety of $X$ is of general type.
\item For every subfield $k_0\subset \mathbb{C}$, every embedding $k_0\to k$, and every variety $X_0$ over $k_0$ with $X\cong X_0\otimes_{k_0} k$, we have that $X_{0,\CC}$ is Brody hyperbolic.  
\end{enumerate}
\end{conjecture}

The original versions of this conjecture appeared in   \cite{Lang2}. 
In   \cite{FaltingsLang}, Faltings proved the above conjecture for projective curves and, more generally, closed subvarieties of abelian varieties. By Faltings's earlier work on the moduli space of abelian varieties \cite{FaltingsComplements}, Conjecture \ref{conj:lang} is also known to hold for projective varieties over $k$ with a  finite morphism to the moduli stack of principally polarized abelian varieties.

 \subsection{Automorphisms}
   Our first result concerns   the finiteness of the automorphism group of a projective arithmetically hyperbolic variety.  
   
 \begin{theorem}\label{thm1}  
 If $X$ is a projective arithmetically hyperbolic variety over $k$, then $\Aut_k(X)$ is finite.
 \end{theorem}
 
 Note that a Brody hyperbolic projective variety has only finitely many automorphisms \cite[Theorem~5.4.4]{Kobayashi}, and that  a projective variety over $k$ of general type has only finitely many automorphisms \cite[\S11]{Iitaka}. Thus, Theorem \ref{thm1} is in accordance with the Green--Griffiths--Lang conjecture (Conjecture \ref{conj:lang}).

  We do not know whether Theorem \ref{thm1}  holds for non-proper (e.g., affine) arithmetically hyperbolic varieties, although it most likely does.      Instead of the finiteness, the next result   verifies the \emph{local finiteness}. Here, we follow standard terminology and say that 
a group $G$ is \emph{locally finite} if every finitely generated subgroup of $G$ is finite. 

\begin{theorem}[Amerik + Bass--Lubotzky]\label{thm:ends} Let $X$ be an   arithmetically hyperbolic variety over $k$. Then,  $\Aut_k(X)$ is a locally finite group and every dominant endomorphism of $X$ is an automorphism of finite order.
\end{theorem}

To prove the second statement  of Theorem \ref{thm:ends} we   use properties of dynamical systems of arithmetically hyperbolic varieties and Amerik's theorem on existence of points with infinite orbit. The first statement  of Theorem \ref{thm:ends} then  follows from combining the second statement with  a well-known theorem of Bass--Lubotzky \cite[Corollary~1.2]{Bass}: if $X$ is a   variety over $k$ and $\Gamma\subset \Aut_k(X)$ is a finitely generated torsion subgroup, then $\Gamma$ is finite.

 However, the proof of Theorem \ref{thm1} is not a mere combination of results of Amerik and Bass--Lubotzky. Indeed, 
in our proof  of Theorem \ref{thm1} we will require the following criterion for the finiteness of the automorphism group of a projective variety.   

 \begin{theorem}  \label{thm:tor_is_fin} Let   $X$ be a projective variety over $k$. If $\Aut_k(X)$  is a  torsion group, then $\Aut_k(X)$ is a  finite group.
 \end{theorem}

Our proof of Theorem \ref{thm:tor_is_fin} uses the theorem of the base, the existence of elements of infinite order in positive-dimensional algebraic groups over $k$, and the well-known fact that automorphisms preserving a fixed ample class in the N\'eron-Severi group form a finite type    group scheme. The projectivity seems to be crucial.  

In \cite[Theorem~1.6]{JXie} we prove an analogous finiteness criterion for the group of birational self-maps \emph{using} Theorem \ref{thm:tor_is_fin}. Namely, we show that, if $X$ is a non-uniruled proper integral variety over $k$ such that the group $\mathrm{Bir}_k(X)$ of birational self-maps $X\dashrightarrow X$ over $k$ is torsion, then $\mathrm{Bir}_k(X)$ is finite.
 
We stress that, when $k=\CC$ and $X$ is smooth, the proof of Theorem \ref{thm:tor_is_fin} is much simpler. Indeed, in this case, one could for instance appeal to \cite[Theorem~2.1]{DHZ} to see that $\Aut_{\CC}(X)$ is finitely generated, so that the desired finiteness result follows directly from Bass--Lubotzky's theorem \cite[Corollary~1.2]{Bass}.

 \subsection{Persistence}
 Our second result concerns the Persistence Conjecture (see also \cite[Conjecture~1.20]{vBJK} or \cite[Conjecture~17.5]{JBook}).
 
 \begin{conjecture}[Persistence Conjecture]\label{conj:pers}
 Let $L/k$ be an extension of algebraically closed fields of characteristic zero. If $X$ is a finite type separated arithmetically hyperbolic scheme over $k$, then $X_L$ is arithmetically hyperbolic over $L$.
 \end{conjecture}
 
 If $X$ is proper over $k$, then the Persistence Conjecture is in fact a consequence of the Green--Griffiths--Lang conjecture. Indeed, if every subvariety of $X$ over $k$ is of general type, then the same holds for the subvarieties of $X_L$.
 
  The second main result of this paper is a general criterion for proving the Persistence Conjecture. We refer the reader to Definition \ref{defn:mb} for the notion of ``mild boundedness''.
  
  \begin{theorem}\label{thm:mb} Let $L/k$ be an extension of algebraically closed fields of characteristic zero. Let $X$ be an arithmetically hyperbolic finite type separated scheme such that, for every algebraically closed subfield $K\subset L$ containing $k$, we have that $X_K$ is mildly bounded over $K$. Then $X_L$ is arithmetically hyperbolic over $L$.
  \end{theorem}
  
  Roughly speaking, Theorem \ref{thm:mb} proves the Persistence Conjecture for varieties which are mildly bounded.  As a first application of Theorem \ref{thm:mb} we verify the Persistence Conjecture for Brody hyperbolic projective varieties.
  
  \begin{theorem}\label{thm:geometricity_intro} Let $k\subset \CC$ be an algebraically closed subfield.
Let $X$ be a projective variety over $k$. Assume that $X_{\CC}$ is Brody hyperbolic. Then $X$ is arithmetically hyperbolic over $k$ if and only if $X_{\CC}$ is arithmetically hyperbolic over $\CC$.
\end{theorem}

  We also use Theorem \ref{thm:mb} to prove the Persistence Conjecture for varieties with a finite map to an abelian variety; see Theorem \ref{thm:yamanoi}. Moreover, 
in a series of papers \cite{vBJK, JLevin, JLitt, JSZ}, we \emph{use} Theorem \ref{thm:mb} to prove the Persistence Conjecture in the following cases:
  
  \begin{enumerate}
  \item  surfaces with non-trivial irregularitiy \cite {vBJK},  
  \item  varieties with a quasi-finite map to a semi-abelian variety \cite{vBJK},
    \item hyperbolically embeddable smooth affine varieties \cite{JLevin},
    \item varieties with a quasi-finite period map \cite{JLitt}, and
    \item moduli spaces of polarized varieties with semi-ample canonical bundle \cite{JSZ};
  \end{enumerate}

\begin{ack} We are most grateful to Michel Brion for  several helpful discussions and comments.   We thank Antoine Chambert-Loir for many helpful comments on an earlier version of our paper. We thank Remy van Dobben de Bruyn for pointing us to Frey-Jarden's theorem \cite{FreyJarden}.
We thank Jason Starr for discussions on automorphisms of groupless varieties and Daniel Loughran for discussions on arithmetically hyperbolic varieties. We thank Ljudmila Kamenova for helpful discussions on algebraic hyperbolicity and hyperk\"ahler varieties.  We thank Ekaterina Amerik  and Ronan Terpereau for helpful discussions on dynamical systems.  
This research was supported through the programme  ``Oberwolfach Leibniz Fellows'' by the Mathematisches Forschungsinstitut Oberwolfach in 2018.
We gratefully acknowledge support from SFB/Transregio 45.
\end{ack}

   \begin{con} Throughout this paper $k$ will be an algebraically closed field of characteristic zero.
 A variety over  $k$ is a finite type separated   integral $k$-scheme.  
 \end{con}

 \section{Groupless varieties}\label{section:groupless}
Following  \cite{JV}, we say that  a variety $X$ over a field $k$ is \emph{groupless} if,  for every  finite type connected group scheme $G$ over $k$, every morphism $G\to X$ is constant. 
 Grouplessness  is sometimes referred to as ``algebraic hyperbolicity'' or ``algebraic Lang hyperbolicity''; see   \cite[Remark~3.2.24]{Kobayashi}. To avoid confusion, we will    only use the term ``algebraically hyperbolic'' for the notion defined by Demailly \cite{Demailly, JKa}.
 
 We will make use of the   following simple lemma; for a proof   we refer to \cite{JKa}.

 \begin{lemma}\label{lem:groupless}
The following statements hold.
 \begin{enumerate}
\item A finite type scheme $X$  over $k$   is groupless over $k$ if and only if  every morphism $\mathbb{G}_{m,k}\to X$ is constant and, for every abelian variety $A$ over $k$, every morphism $A\to X$ is constant.
 \item  A proper scheme $X$ over $k$ is groupless over $k$ if and only if, for every abelian variety $A$ over $k$, every morphism of varieties $A\to X$ is constant.  
 \item   Let $X$ be a  proper groupless scheme    over $k$. Then, for every smooth variety $S$ over $k$ and every dense open $U\subset S$, we have that any morphism $U\to X$ extends uniquely to a morphism $S\to X$.
\end{enumerate}
\end{lemma}
%
%
%

\section{Arithmetic hyperbolicity}\label{section:ar_hyp}
Central to  this paper is the  notion of being arithmetically hyperbolic (defined in the introduction).
We unravel what  arithmetic hyperbolicity entails for affine varieties. To do so, let $X$ be an affine variety over $k$. Choose integers $n\geq 1$ and $m\geq 1$,   choose polynomials $f_1,\ldots, f_n \in k[x_1,\ldots,x_m]$, and choose an isomorphism $$X\cong \Spec(k[x_1,\ldots,x_m]/(f_1,\ldots,f_n)).$$ Let $A_0$ be the subring of $k$ generated by the (finitely many) coefficients of the polynomials $f_1,\ldots, f_n$. Note that $A_0\subset k$ is a $\ZZ$-finitely generated subring. Define $$\mathcal{X}:=\Spec (A_0[x_1,\ldots, x_m]/(f_1,\ldots,f_n))$$ and note that $\mathcal{X}_k \cong X$. Now, the following statements are equivalent.
\begin{enumerate}
\item The variety $X$ is arithmetically hyperbolic over $k$.
\item For every   $\ZZ$-finitely generated subring $A\subset k$ containing $A_0$, the set
\[
\{(a_1,\ldots,a_m)\in A^m \ | \ f_1(a_1,\ldots,a_m) = \ldots = f_n(a_1,\ldots,a_m) =0\}
\] is finite.
\end{enumerate}
 Thus, roughly speaking, one could say that an algebraic variety over $k$ is arithmetically hyperbolic over $k$ if it has only finitely many ``$A$-valued points'', for any choice of finitely generated subring $A\subset k$. 
    Roughly speaking,  to test the arithmetic hyperbolicity of a variety,   one has to choose a model  and ``check'' the finiteness of integral points on every $\ZZ$-finitely generated subring of $k$.

 We conclude this section with examples of arithmetically hyperbolic varieties.

 \begin{example}[Closed subvarieties of abelian varieties]\label{examples} It follows from Faltings's theorem \cite{FaltingsComplements} that 
 a one-dimensional variety $X$ over $k$ is arithmetically hyperbolic over $k$ if and only if it is groupless over $k$.
Moreover,  it follows from Faltings's theorem   \cite{FaltingsLang} that a closed subvariety $X$  of an abelian variety $A$ over $k$ is arithmetically hyperbolic over $k$ if and only if $X$ is groupless.   
\end{example}
 
 \begin{example}[Moduli spaces] Let $g\geq 1$ be an integer, let $n>2$ be an integer, and let $\mathcal{A}_g^n$ be the fine moduli space of $g$-dimensional principally polarized abelian schemes with full level $n$ structure over $k$. By Faltings's theorem (formerly the Shafarevich conjecture), the variety  $\mathcal{A}_g^n$ is an arithmetically hyperbolic smooth quasi-projective variety over $k$. In particular, if $X$ is a  variety over $k$ which admits  a 
  a quasi-finite morphism to $\mathcal{A}_g^n$, then $X$ is arithmetically hyperbolic over $k$. Applications of Faltings's result are given in \cite{JL, JLFano} to moduli spaces of Fano threefolds and low degree smooth hypersurfaces. 
\end{example}
 
 \begin{example}[Lawrence--Sawin] Let $A$ be an abelian variety over $\Qbar$ of dimension at least four.
 Following the strategy of Lawrence--Venkatesh \cite{LawrenceVenkatesh}, it is shown in \cite{LawrenceSawin}   that certain moduli spaces of smooth hypersurfaces in $A$ are arithmetically hyperbolic over $\Qbar$.
 \end{example}

 \begin{example}
 Standard conjectures on Calabi-Yau varieties imply that    a Calabi-Yau variety is not arithmetically hyperbolic. For example,  one can show that a hyperkähler variety with Picard rank at least three is not arithmetically hyperbolic; see Theorem \ref{thm:hyp_kah_intro}.
 \end{example}

\subsection{Integral points   on abelian varieties}\label{section:ht} In this section we   show that arithmetically hyperbolic varieties are groupless (Proposition \ref{prop:ar_is_gr}). To prove this result,  recall that, for $K$  a finitely generated field of characteristic zero and   $A$   an abelian variety over $K$, the abelian group $A(K)$ is finitely generated  \cite[Cor.~7.2]{ConradTrace}.
The rank of  $A(K)$   can grow arbitrarily large over finite extensions of $K$, as was shown by Frey-Jarden \cite{FreyJarden}.

\begin{lemma}[Frey-Jarden's rank jumping]\label{lem:frey} Let $K$ be a finitely generated field of characteristic zero and let $A$ be an abelian variety over $K$. Then there is a finite field extension $L/K$ such that the rank of $A(L)$ is strictly bigger than the rank of $A(K)$.
\end{lemma}

\begin{remark}[Hassett-Tschinkel's theorem] Let $K$ be a  field of characteristic zero and let $A$ be an abelian variety over $K$. 
Frey-Jarden's theorem implies that there is a finite field extension $L/K$  such  that $A(L)$ is dense in $A$.  Hassett-Tschinkel proved the   stronger statement that, there is a finite field   extension $L/K$ and a point $P$ in $A(L)$ such that  the subgroup generated by $P$ in $A(L)$ is Zariski dense in $A$. Indeed, 
 to prove this, we may and do assume that $K$ is a finitely generated field (of characteristic zero). Then, the statement of the proposition follows from the proof of     \cite[Prop~3.1]{HassettTschinkel}. (In \emph{loc. cit.}, the authors assume $K$ is a number field, but their proof works for every finitely generated field of characteristic zero.)
 \end{remark}
 
 \begin{corollary}\label{cor:pot_den0}
 Let $k$ be an algebraically closed field of characteristic zero and let $G$ be an abelian variety over $k$. Then there is a finitely generated subfield  $L\subset k$ and an abelian variety $\mathcal{G}$ over $L$ with $\mathcal{G}_k \cong G$ over $k$ such that $\mathcal{G}(L)$ is  Zariski-dense  in $G$. 
 \end{corollary}
 \begin{proof} We first ``descend'' the abelian variety $G$ over $k$ to a finitely generated subfield. Thus, choose a finitely generated subfield $K\subset k$ and an abelian variety $\mathcal{G}'$ over $K$ such that $\mathcal{G}'_k\cong G$ over $K$. By  Lemma \ref{lem:frey}, there is a finite field extension $L$ of $K$ contained in $k$ such that $\mathcal{G}'(L)$ is Zariski dense in $G$. Thus, the corollary holds with $\mathcal{G}:=\mathcal{G}'_L$.
 \end{proof}

\begin{corollary}\label{cor:pot_den} Let $k$ be an algebraically closed field of characteristic zero and let $G$ be an abelian variety over $k$. Then there is a smooth $\ZZ$-finitely generated subring $A\subset k$ and an abelian scheme $\mathcal{G}\to \Spec A$ with $\mathcal{G}_k \cong G$ such that $\mathcal{G}(A)$ is  Zariski-dense  in $G$. 
\end{corollary}
\begin{proof}
Choose a smooth $\ZZ$-finitely generated subring $A \subset k$ and an abelian scheme $\mathcal{G}\to \Spec A$ such that $\mathcal{G}_{k}\cong G$ and such that $\mathcal{G}(\mathrm{Frac}(A))$ is Zariski dense in $G$; such data exists by Corollary \ref{cor:pot_den0}.   To conclude the proof,  note that the geometric fibres of $\mathcal{G}\to S$ contain no rational curves, so that $\mathcal{G}(A) = \mathcal{G}(\mathrm{Frac}(A))$ by  \cite[Proposition~6.2]{GLL}.
\end{proof}

 \begin{proposition}\label{prop:ar_is_gr}
 If $X$ is an arithmetically hyperbolic variety over $k$, then $X$ is groupless over $k$.
 \end{proposition}
 \begin{proof}
We first show that every morphism $f:\mathbb{G}_{m,k}\to X$ is constant. To do so, choose a $\ZZ$-finitely generated subring $A\subset k$, a model $\mathcal{X}$ for $X$ over $A$ and a morphism $F:\mathbb{G}_{m,A} \to \mathcal{X}$ with $F_k\cong f$ such that      $\mathbb{G}_m(A)$ is infinite. It follows that $\mathcal{X}(A)$ is infinite, unless $f$ is constant.

 Now, let $G$ be an abelian variety over $k$, and let $G\to X$ be a morphism.   
 To  show that $G\to X$ is constant, choose a smooth $\ZZ$-finitely generated subring $A\subset k$, a model $\mathcal{X}$ for $X$ over $A$, an abelian scheme $\mathcal{G}$ over $A$ with $\mathcal{G}_k\cong G$, and a morphism $\mathcal{G}\to \mathcal{X}$ such that $\mathcal{G}(A)$ is Zariski dense in $G$; such data  exists by Corollary \ref{cor:pot_den}.    Note that the set $ \mathcal{G}(A)$ maps to the set $ \mathcal{X}(A)$ via $\mathcal{G}\to\mathcal{X}$. Therefore, since  $\mathcal{G}(A)$  is Zariski dense in $G$, the image of the finite set $\mathcal{G}(A)$ in $\mathcal{X}(A)$  is Zariski dense in  the (closed, scheme-theoretic) image of  $G\to X$. However, since $X$ is arithmetically hyperbolic over $k$, any   closed subscheme of $X$ is arithmetically hyperbolic over $k$. Thus, the image of $G\to X$ is an arithmetically hyperbolic connected variety whose set of $k$-points contains a finite and dense subset. This implies that the image of $G\to X$ is finite, so that  $G\to X$ is constant.
  To conclude the proof,  apply  the first part of  Lemma \ref{lem:groupless}.
 \end{proof}

%
%
  
  \section{Persistence of   arithmetic hyperbolicity}\label{section:geometricity}
  In this section we study the persistence of arithmetic hyperbolicity over field extensions, under suitable ``boundedness'' assumptions related to Demailly's notion of algebraic hyperbolicity \cite{Demailly}.
  
   Our main result (Theorem \ref{thm:geometricity_intro}) says that arithmetic hyperbolicity  of a projective variety persists over field extensions provided that the variety is Brody hyperbolic.

\subsection{Mild boundedness}

With the aim of isolating the weakest     property we require for proving the persistence of arithmetic hyperbolicity along field extensions, we start with the notion of ``mild boundedness''.

  \begin{definition}\label{defn:mb}
  A finite type scheme $X$ over $k$ is \emph{mildly bounded} if, for every smooth quasi-projective curve $C$ over $k$, there exist an integer $m\geq 1$ and points $c_1,\ldots,c_m\in C(k)$ such that, for every $x_1,\ldots,x_m\in X(k)$ the set 
  \[
  \Hom_k((C,c_1,\ldots,c_m),(X,x_1,\ldots,x_m)) := \{f:C\to X \ | \ f(c_1) = x_1, \ldots, f(c_m) = x_m\}
  \] is finite.
  \end{definition}

 \begin{example}\label{ex}
 If $X$ is a smooth quasi-projective connected curve over $k$, then $X$ is mildly bounded if and only if $X\not\cong \mathbb{A}^1_k$ and $X\not\cong \mathbb{P}^1_k$; see \cite[Corollary~5.13]{vBJK}.  
 \end{example}

 \begin{lemma}\label{lem1}
 Let $k\subset L$ be an extension of algebraically closed fields of characteristic zero such that $L$ is of transcendence degree $1$ over $k$.  If $X$ is an arithmetically hyperbolic   mildly bounded variety  over $k$, then $X_L$ is arithmetically hyperbolic over $L$.  
 \end{lemma}
 \begin{proof}   
Let $A\subset k$ be a $\ZZ$-finitely generated subring and let $\mathcal{X}\to \Spec A$ be a finite type separated model for $X$ over $A$ (so that $\mathcal{X}_k \cong X$). Note that $\mathcal{X}$ is also a finite type model for $X_L$ over $A\subset L$. 

  To prove the lemma, we choose    a $\ZZ$-finitely generated subring  $A\subset B\subset L$  $B$ of $L$ containing $A$ with $\Spec B \to \Spec \mathbb{Z}$ a smooth morphism.   Note that it suffices to show that $\mathcal{X}(B)$ is finite.   We define $\mathcal{C} := \Spec B$, and note that $\mathcal{X}(\mathcal{C}) = \mathcal{X}(B)$.  

In case $B$ is actually contained in $k$, we are done. Indeed,  since $X$ is arithmetically hyperbolic over $k$, for any  $\mathbb{Z}$-finitely generated subring $ A'\subset k$ of $k$ containing $A$, the set $\mathcal{X}(A')$ is finite. Therefore, to prove that $\mathcal{X}(\mathcal{C})$ is finite, we may and  do assume that  the subring   $  B\subset L$ is not contained in $k$. 

 Let $K$ be the fraction field of $A$, and note that $K\subset k\subset L$. Then, as $B$ is not contained in $k$ and $L$ is of transcendence degree one over $k$, we have that $\mathcal{C}_K\to \Spec K$ is a smooth affine connected one-dimensional scheme over $K$.  

Define $C:= \mathcal{C}_k$.  Note that $C$ is a smooth affine   curve over $k$, and that there is an inclusion of sets 
\[ 
\mathcal{X}(\mathcal{C})=\Hom_A(\mathcal{C},\mathcal{X}) \subset \Hom_k(C,X).
\]

We now use that $X$ is mildly bounded over $k$ to show that $\mathcal{X}(\mathcal{C})$ is finite. Indeed, since $X$ is mildly bounded over $k$, there exists an integer $m\geq 1$ and points $c_1,\ldots,c_m\in C(k)$ such that, for all $x_1, \ldots, x_m\in X(k)$, the set $$\Hom_k((C,c_1,\ldots,c_m),(X,x_1,\ldots,x_m))$$ is finite. We choose $m$ and $c_1,\ldots,c_m \in C(k)$ with this property.

We now choose  a   $\ZZ$-finitely generated subring $ A'\subset k$   of $k$ containing $A$ and   $\overline{c}_1, \ldots, \overline{c}_m \in \mathcal{C}(A')$   such that     $\Spec A'\to \Spec \ZZ$ is smooth,   and  $\overline{c}_{1,k} = c_1, \ldots, \overline{c}_{m,k} = c_m$ in $C(k)$.    (In other words, we extend the base ring $A$ in such a way that the points $c_1, \ldots, c_m$ become sections of $\mathcal{C}$.)

Now,   define $\mathcal{D}:= \mathcal{C}\times_A A'$. Note that $\mathcal{X}(\mathcal{C}) \subset \mathcal{X}(\mathcal{D})$.  Thus, it suffices to show that $\mathcal{X}(\mathcal{D})$ is finite. Note that $\overline{c}_1,\ldots, \overline{c}_m$ are sections of $\mathcal{D}\to \Spec A'$.
Moreover, 
if $f:\mathcal{D} \to \mathcal{X}$ is  an element of $\mathcal{X}(\mathcal{D})$, then $f(\overline{c}_i) \in \mathcal{X}(A')$. Therefore, we have an inclusion of sets 
\[ \mathcal{X}(\mathcal{D}) \subset \bigcup_{(\overline{x}_1,\ldots, \overline{x}_m) \in \mathcal{X}(A')^m} \Hom((\mathcal{D}, \overline{c}_1,\ldots,\overline{c}_m), (\mathcal{X}, \overline{x}_1,  \ldots, \overline{x}_m)), \]
 where $\mathcal{X}(A')^m$ denotes the product of sets $\mathcal{X}(A')\times \ldots \times \mathcal{X}(A')$. 
 
Note that $\mathcal{D}_k = C$. Thus, for any $ (\overline{x}_1,\ldots, \overline{x}_m) \in \mathcal{X}(A')^m$, we have an inclusion of sets
 \[
  \Hom_{A}((\mathcal{D}, \overline{c}_1,\ldots,\overline{c}_m), (\mathcal{X}, \overline{x}_1, \ldots, \overline{x}_m)) \subset \Hom_k ( (C, c_1,\ldots,c_m), (X,\overline{x}_{1,k},\ldots, \overline{x}_{m,k})).
 \]
Thus, we conclude that  $\mathcal{X}(\mathcal{D})$ is finite from the finiteness of $\mathcal{X}(A')^m$ and the finiteness of the set $\Hom_k ( (C, c_1,\ldots,c_m), (X,\overline{x}_{1,k},\ldots, \overline{x}_{m,k}))$. This concludes the proof.
 \end{proof}

 \begin{lemma}\label{lem2}
  Let $k\subset L$ be an extension of algebraically closed fields of characteristic zero such that $L$ has finite transcendence degree  over $k$. Let $X$ be an arithmetically hyperbolic finite type separated scheme over $k$. Assume that, for every algebraically closed subfield $k\subset K\subset L$, we have that  $X_K$ is mildly bounded over $K$.  Then $X_L$ is arithmetically hyperbolic over $L$.
 \end{lemma}
 \begin{proof}  We proceed by induction on the transcendence degree of $L$ over $k$. Let $n:= \mathrm{trdeg}_k(L)$.   If $n=0$, we are done by our assumption that $X$ is arithmetically hyperbolic over $k$ (as $k=L$ in this case). Thus, assume that $n>0$, and  let $k\subset K\subset L$ be an algebraically closed subfield with $\mathrm{trdeg}_k(K) =n-1$. Note that $X_K$ is mildly bounded over $K$ by assumption. In particular, by the induction hypothesis (and the fact that, for every algebraically closed subfield $k\subset K' \subset K$ we have that $X_{K'}$ is mildly bounded over $K'$), we see that  $X\otimes_k K$ is arithmetically hyperbolic over $K$. Finally,   as $K\subset L$ has transcendence degree $1$ and $X_K$ is a mildly bounded    arithmetically hyperbolic variety over $K$ (by assumption), the result follows from Lemma \ref{lem1}.
 \end{proof}
 
 \begin{proof} [Proof of Theorem \ref{thm:mb}]
To say that $X_L$ is arithmetically hyperbolic over $L$ is equivalent to saying that, for every algebraically closed subfield $k\subset K\subset L$ of finite transcendence degree over $k$, the variety $X_K$ is arithmetically hyperbolic over $K$. Therefore, the result follows from Lemma \ref{lem2}. 
 \end{proof}
 
Let us briefly discuss the theorem of Siegel-Mahler-Lang and illustrate how Theorem \ref{thm:mb} can be used in practice. 
 
 \begin{remark}[Siegel-Mahler-Lang's theorem]  
 If $A$ is  a $\ZZ$-finitely generated integral domain of characteristic zero, then the theorem of Siegel-Mahler-Lang says that $(\mathbb{A}^1_{\ZZ}\setminus \{0,1\})(A)$ is finite; see \cite{LangIHES}. 
Let us show how one can deduce the Mahler-Lang theorem from the particular case of Siegel's theorem.  Indeed, first, one proves the desired finiteness statement when $\dim A = 1$, i.e., one shows that  $\mathbb{A}^1_{\Qbar}\setminus \{0,1\}$ is arithmetically hyperbolic over $\Qbar$. Now, to prove the corresponding statement for finitely generated rings of higher dimension, note that for every algebraically closed field $k$ of characteristic zero, the variety  $\mathbb{A}^1_k\setminus\{0,1\}$ is mildly bounded over $k$ (see Example \ref{ex}), so that it is arithmetically hyperbolic over $k$ by Theorem \ref{thm:mb}. This precisely means that, for every $\ZZ$-finitely generated integral domain of characteristic zero, the set $(\mathbb{A}^1_{\ZZ}\setminus \{0,1\})(A)$ is finite. 
 \end{remark}

 \subsection{A conjecture on mildly bounded varieties} It would be interesting to show that arithmetically hyperbolic varieties are actually mildly bounded, as this would imply that arithmetically hyperbolic varieties remain arithmetically hyperbolic over any extension of the base field.   
 
  \begin{conjecture}\label{conj:mb}
  If $X$ is an arithmetically hyperbolic variety over $k$, then $X$ is mildly bounded over $k$.
  \end{conjecture}
 
 Mild boundedness is  essentially  the ``weakest'' notion of boundedness required for arithmetic hyperbolicity to persist over a field extension, and Conjecture \ref{conj:mb} predicts that this ``weak'' notion of boundedness holds for all arithmetically hyperbolic varieties.  Note that semi-abelian varieties are mildly bounded (see \cite[Proposition~1.9]{vBJK}). In particular, a mildly bounded projective variety might be non-arithmetically hyperbolic (even not of general type).  This makes the notion of mild boundedness easier to verify than other notions of boundedness.

 \subsection{Algebraic hyperbolicity, boundedness, Brody hyperbolicity}\label{section:alg_hyp}  
 In this section we combine our  results from our earlier work with Kamenova \cite{JKa} with Theorem \ref{thm:mb}, and  provide  new results  on the persistence of arithmetic hyperbolicity for projective varieties. 
We start with defining the following three notions of boundedness introduced in \cite[\S4]{JKa}.

A projective variety $X$ over $k$ is \emph{algebraically hyperbolic over $k$} if, for every ample line bundle $L$ on $X$, there is a real number $\alpha(L)$ such that, for every smooth projective connected curve $C$ and every morphism $f:C\to X$, the inequality \[
\deg_C f^\ast L \leq \alpha(L)\cdot  \mathrm{genus}(C)
\] holds. Moreover,  a projective variety $X$ over $k$ is \emph{1-bounded over $k$}  if, for every smooth projective connected curve $C$ over $k$, the  scheme $\underline{\Hom}_k(C,X)$ is of finite type over $k$.
 Similarly,  a projective variety $X$ over $k$ is \emph{$(1,1)$-bounded over $k$} (or: \emph{geometrically hyperbolic over $k$}) if, for every   smooth projective connected curve $C$ over $k$, every $c$ in $C(k)$, and every $x$ in $X(k)$, the scheme
  \[
  \underline{\Hom}_k((C,c), (X,x))  
  \] parametrizing morphisms $f:C\to X$ with  $f(c) =x$ is of  finite type over $k$.

 Note that algebraic hyperbolicity implies $1$-boundedness, and that $1$-boundedness implies $(1,1)$-boundedness. However, 
we stress that mild boundedness is \textbf{strictly} weaker than the above notions. Indeed, a non-zero abelian variety is mildly bounded  \cite[Proposition~1.9]{vBJK}, although it is \textbf{neither} algebraically hyperbolic,   $1$-bounded,   nor $(1,1)$-bounded.

  \begin{theorem}\label{thm:geometricity3} Let $k\subset L$ be an extension of algebraically closed fields of characteristic zero.
  Let $X$ be a projective  arithmetically hyperbolic variety over $k$ which is algebraically hyperbolic over $k$, or $1$-bounded over $k$, or $(1,1)$-bounded over $k$. Then $X_L$ is arithmetically hyperbolic over $L$.
  \end{theorem}
 \begin{proof} Since algebraic hyperbolicity implies $1$-boundedness and $(1,1)$-boundedness implies $1$-boundedness, we may assume that $X$ is  $(1,1)$-bounded over $k$. Then 
 $X_L$ is  $(1,1)$-bounded  \cite[Theorem~7.2]{JKa}.   In particular, for every algebraically closed subfield $K\subset L$ containing $k$, the projective variety $X_K$ is $(1,1)$-bounded over $K$, and thus mildly bounded over $K$ (see for example \cite[Lemma~4.6]{JKa}).   By Theorem \ref{thm:mb}, we   conclude that $X_L$ is arithmetically hyperbolic over $L$.  
 \end{proof}

\begin{theorem}[Yamanoi]\label{thm:nwy}  
Let $X$ be a  smooth projective connected variety whose Albanese map $X\to \mathrm{Alb}(X)$ is finite.   If   $X$ is groupless, then   $X$ is algebraically hyperbolic over $k$.
\end{theorem}
\begin{proof} Since algebraic hyperbolicity  is compatible with extensions of  the base field \cite[Theorem~7.1]{JKa}, we may and do assume that $k=\CC$ in which case the result is due to Yamanoi   \cite{Yamanoi}.  
\end{proof}

\begin{theorem} \label{thm:yamanoi} Let $k\subset L$ be an extension of algebraically closed fields of characteristic zero. If $X$ is a smooth projective connected arithmetically hyperbolic variety over $k$ whose Albanese map is finite, then $X_L$ is arithmetically hyperbolic over $L$.
\end{theorem}

\begin{proof} Since $X$ is arithmetically hyperbolic over $k$, it is groupless over $k$ (Proposition \ref{prop:ar_is_gr}). Therefore, 
it follows from  Theorem \ref{thm:nwy} that $X$ is algebraically  hyperbolic, so that  the result follows from Theorem \ref{thm:geometricity3}.  
\end{proof}

We deduce Theorem \ref{thm:geometricity_intro} from the following   more general result.

\begin{proposition}\label{prop:geometricity22} Let $k\subset L$ be an extension of algebraically closed fields of characteristic zero. 
Let $X$ be a projective  arithmetically hyperbolic variety over $k$. Suppose that there is a subfield $k_0\subset k$, a projective variety $X_0$ over $k_0$, an isomorphism $X_{0,k}\cong X$ over $k$, and an embedding $k_0\subset \CC$ such that $X_{0,\CC}$ is Brody hyperbolic. Then $X_L$ is arithmetically hyperbolic over $L$.
\end{proposition}
\begin{proof}
Since $X_{0,\CC}$ is a projective Brody hyperbolic variety, it follows from Brody's lemma that $X_{0,\CC}$ is a Kobayashi hyperbolic variety \cite[Theorem~3.6.3]{Kobayashi}. In particular, as $X_{0,\CC}$ is a Kobayashi hyperbolic projective variety, it follows that $X_{0,\CC}$ is algebraically hyperbolic over $\CC$; see \cite[Theorem~1.2]{JKa}. It follows that $X_0$ is algebraically hyperbolic over $k_0$. However, by the compatibility of algebraic hyperbolicity with extensions of the base field \cite[Theorem~7.1]{JKa}, we conclude that $X_{0,k}\cong X$ is algebraically hyperbolic over $k$. Now, as $X$ is an algebraically hyperbolic and arithmetically hyperbolic projective variety over $k$, we conclude that $X_L$ is arithmetically hyperbolic (Theorem \ref{thm:geometricity3}).
\end{proof}
\begin{proof}[Proof of Theorem \ref{thm:geometricity_intro}] Let $X$ be a projective arithmetically hyperbolic variety over $k$ such that $X_{\CC}$ is Brody hyperbolic. It follows that $X_{\CC}$ is arithmetically hyperbolic over $\CC$ from Proposition \ref{prop:geometricity22} with $k_0 := k$ and $L:=\CC$.
\end{proof}


  \section{Torsion automorphism groups are finite}\label{section:tor_is_fin}
  The main result of this section says that, for a projective variety $X$ over the field $k$ (of characteristic zero), the automorphism group of $X$ is infinite if and only if it is non-torsion; see Theorem \ref{thm:tor_is_fin}. To prove this result, we will use basic facts about N\'eron-Severi groups, automorphisms preserving some fixed ample class, and $k$-points of positive-dimensional finite type group schemes over $k$. Presumably this result is ``well-known'' to experts, as the arguments we use   already appear   in the literature (in some form or another); see for instance  \cite{Cantat}, \cite[Corollary~6.1.7]{LittLesieutre}, or \cite{ZhangDeQi}.

  Let $S$ be a scheme and let $X\to S$ be a morphism. The functor $\mathrm{Aut}_{X/S}$ on the category of schemes over $S$ is defined  by $\mathrm{Aut}_{X/S}(T) = \mathrm{Aut}_T(X_T)$. We first use basic representability results for this functor to prove the following result.

   \begin{lemma}\label{lem:finiteness_of_aut} Let $k\subset L$ be an extension of algebraically closed fields. Let $X$ be a projective variety over $k$.
 If $\Aut_k(X)$ is finite, then $\Aut_L(X_L)$ is finite and $\Aut_k(X) = \Aut_L(X_L)$.
 \end{lemma}
 \begin{proof} 
 The  group scheme $\mathrm{Aut}_{X/k}$ is locally of finite type over $k$ and has only finitely many $k$-points.   This implies that $\mathrm{Aut}_{X/k}$ is finite over $k$.  Thus, $\mathrm{Aut}_{X_L/L} = \mathrm{Aut}_{X/k}\otimes_k L$ is finite over $L$, so that  $\Aut_L(X_L) = \Aut_{X_L/L}(L)$ is finite.  This proves the lemma. 
 \end{proof}
  
 For $X$ a proper   scheme over a field $k$, we let $\NS(X) $ be the N\'eron-Severi group of $X$, and we define $\NS(X)_{\QQ} := \NS(X)\otimes_\ZZ \QQ$.  If $L$ is a line bundle on $X$, we let $[L]$ denote the class of $L$ in $\NS(X)_{\QQ}$.
The following well-known proposition  says that, for $L$ an ample line bundle on a projective variety $X$ over $k$, the group of automorphisms of $X$ over $k$ which fix the class of $L$ in $\NS(X)_{\QQ}$ is the group of $k$-points on a finite type group scheme over $k$.

  \begin{proposition}\label{prop:rep} 
  Let $X$ be a projective    scheme over $k$, and let $L$ be an ample line bundle on $X$. The  functor  defined by
  \[
  (\mathrm{Sch}/k)^{\mathrm{op}} \to \mathrm{Groups},\]  \[ S \mapsto \{g \in \mathrm{Aut}_S(X_S) \ | \ \mathrm{for \ all \ geometric \ points}  \ \overline{s}\to S,  \ g_{\overline{s}}^\ast [L] = [L] \mathrm{\ in \ } \NS(X_{\overline{s}})_{\QQ} \}
  \] is representable by a finite type group scheme $\Aut_{X/k,[L]}$ over $k$.
  \end{proposition}
  \begin{proof} This is certainly well-known. A  proof of this is given in
    \cite[Remark~2.6]{Zhang}. (Note that we do not need that $k$ is of characteristic zero.)   
  \end{proof}

  We will also require the following simple group-theoretic lemma. It is essentially a   consequence of the fact that a homomorphism of groups $G\to H$ is trivial, provided $G$ is torsion and $H$ is torsion free.
  
 \begin{lemma}\label{lem:gr}
Let $G$ be a torsion group. Let $\Gamma$ be a finitely generated abelian group. Then, any morphism of groups $G\to \mathrm{Aut}(\Gamma)$ has finite image.
 \end{lemma}
 \begin{proof}  Let $G\to \mathrm{Aut}(\Gamma)$ be a morphism and let  $G'$ be its image. Note that $G'$ is a torsion subgroup of $\mathrm{Aut}(\Gamma)$.
 
Recall that, for all positive integers $n$, the group $\mathrm{GL}_n(\ZZ)$ has a normal torsion-free       subgroup of finite index (e.g., the subgroup of matrices which are congruent to the identity modulo $3$).
Therefore, 
 the group   $\mathrm{Aut}(\Gamma)$ has a torsion-free normal finite index subgroup, say $H$. 

Consider the morphism $G'\subset \mathrm{Aut}(\Gamma)\to \mathrm{Aut}(\Gamma)/H$. The kernel of this morphism is $G'\cap H$. Since $H$ is torsion-free and $G'$ is torsion, we see that $G'\cap H$ is trivial. Since the index of  $G'\cap H$ in $G$ is bounded by the index of $H$ in $\mathrm{Aut}(\Gamma)$, we see that $G'\cap H$ has finite index in $G'$. Thus, as  $G'\cap H$ is trivial and of finite index in $G'$, we conclude that $G' = \mathrm{Im}[G\to \Aut(\Gamma)]$ is finite.
 \end{proof}
 
 We now show that  positive-dimensional algebraic groups over algebraically closed fields of characteristic zero have elements of infinite order. This is a non-trivial fact when $k$ is countable.
  \begin{lemma}\label{lem:ft_tor_gr_is_fin} Let $k$ be an algebraically closed field of characteristic zero.
  Let $G$ be a finite type group scheme over $k$ such that $G(k)$ is torsion. Then $G$ is finite.
  \end{lemma}
  \begin{proof} We may and do assume that $G$ is connected.   Since $k$ is of characteristic zero, by Cartier's theorem \cite[Tag~047N]{stacks-project}, the group scheme $G$ is smooth. Thus, $G$ is an ``algebraic group over $k$'' in the sense of \cite{ConradChevalley}.
 By Chevalley's theorem \cite[Theorem~1.1]{ConradChevalley}, there is a  unique  normal affine     connected linear algebraic subgroup $H$ in $G$ such that $G/H$ is an abelian variety. If $H$ is non-trivial, then $H$ contains either  $\mathbb{G}_{a,k}$ or  $\mathbb{G}_{m,k}$ as a subgroup. Since $k$ is of characteristic zero, the group $\mathbb{G}_a(k)$ is not torsion, and the group $\mathbb{G}_m(k)$ is not torsion. Thus, if $H$ is non-trivial, then $H(k)$ contains non-torsion elements. Therefore, as $G(k)$ is torsion, it follows that $H$ is trivial, so that $G$ is an abelian variety (by the defining property of $H$).  However, as $k$ is of characteristic zero, if $G$ is a positive-dimensional abelian variety over $k$, then $G(k)$ contains a point of infinite order by Frey-Jarden's theorem (Lemma \ref{lem:frey}). Therefore, we conclude that $G$ is the trivial group, as required.
  \end{proof}

 \begin{remark}
 Note that, if $L$ is a field of characteristic $p>0$, then $\mathbb{G}_{a,L}$ is a positive-dimensional (non-finite) group scheme over $L$, and  $\mathbb{G}_a(L) = (L,+)$ is an   abelian $p$-torsion group. Thus, Lemma \ref{lem:ft_tor_gr_is_fin} is false over any algebraically closed field $L$ of positive characteristic.  
 \end{remark}

We are now ready to prove the ``criterion'' for finiteness of the automorphism group of a projective variety over an algebraically closed field of characteristic zero.

 \begin{proof}[Proof of Theorem \ref{thm:tor_is_fin}] Let $X$ be a projective variety over $k$ such that $\Aut_k(X)$ is a torsion group.    Since $\Gamma:=\NS(X)$ is  a finitely generated abelian group (see for instance \cite[Theorem~8.4.7]{BLR}), it follows from Lemma \ref{lem:gr} that there is a finite index subgroup $H\subset \Aut_k(X)$ which acts  trivially on   $\Gamma = \NS(X)$.  Let $L$ be an ample line bundle on $X$, and note that $H$ fixes the class of $L$ in $\NS(X)_{\QQ}$.  By Proposition \ref{prop:rep}, the group of automorphisms which leave the class of $L$ fixed in 
 $\NS(X)_{\QQ}$ is representable by a finite type group scheme  $G:=\mathrm{Aut}_{X/k,[L]}$.  Note that $$H\subset G(k)\subset \Aut_k(X).$$ In particular, as $\Aut_k(X)$ is torsion (by assumption), it follows that $G(k)$ is torsion. Thus, by Lemma \ref{lem:ft_tor_gr_is_fin},  the finite type group scheme $G$ is finite over $k$. As $H \subset G(k)$,
  we see that $H $ is finite. Since $H$ is of finite index in $\Aut_k(X)$, we conclude that $\Aut_k(X)$ is finite.
 \end{proof}
 
  \begin{remark}
The analogue of  Theorem \ref{thm:tor_is_fin} is false for projective varieties over $\overline{\mathbb{F}_p}$.   Indeed, let $X$ be a smooth proper connected curve of genus one over $K=\overline{\mathbb{F}_p}$. Then, the automorphism group $\Aut_{K}(X)$   is torsion and infinite. 
\end{remark}

\begin{remark} The analogue of Theorem \ref{thm:tor_is_fin} fails over \emph{any} algebraically closed field of positive characteristic. Indeed,  at the bottom of page 10 in \cite{Brion1}, Brion constructs a smooth projective surface $S$ over $k$ such that $\Aut^0_{S/k}=\mathbb{G}_{a,k}$. In particular, for this surface $S$, the group $\Aut^0_{S/k}(k) = \mathbb{G}_a(k)$ is infinite and torsion. 
  \end{remark}

 \begin{remark}
 Theorem \ref{thm:tor_is_fin} confirms that, if $X$ is a projective variety over $k$ and $\Aut_k(X)$ is torsion, then $\Aut_k(X)$ is \emph{\textbf{finite}}. We stress that this is \emph{\textbf{not}} a consequence of Bass--Lubotzky's theorem which in this case ``only'' says that every finitely generated subgroup of $\Aut_k(X)$ is finite (i.e., $\Aut_k(X)$ is locally finite); see \cite{Bass}. (It also seems worthwhile stressing that there are smooth projective varieties over $\CC$ such that $\Aut_{\CC}(X)$ is a discrete non-finitely generated group; see Remark \ref{remark:les}.)
 \end{remark}
 
As an application of Theorem \ref{thm:tor_is_fin}, we now prove the following more general result.

 \begin{corollary} \label{cor:exis_of_aut0} Let $k\subset L$ be an extension of algebraically closed fields of characteristic zero.
 Let $X$ be a projective variety over $k$. If $\Aut_k(X)$ is torsion, then $\Aut_L(X_L)$ is finite.  
 \end{corollary}
 \begin{proof} 
  By Theorem \ref{thm:tor_is_fin}, the group $\Aut_k(X)$ is finite. Thus, by Lemma \ref{lem:finiteness_of_aut}, the group $\Aut_L(X_L)$ is finite. This proves the corollary.  
 \end{proof}

 \begin{corollary} \label{cor:exis_of_aut} Let $k\subset L$ be an extension of algebraically closed fields of characteristic zero.
 Let $X$ be a projective variety over $k$. Then $X$ has an automorphism of infinite order if and only if $X_L$ has an automorphism of infinite order.
 \end{corollary}
 \begin{proof}
 If $X$ has an automorphism of infinite order, then $X_L$ has an automorphism of infinite order. Thus, to prove the corollary, 
 suppose that $X$ has no automorphism of infinite order. Then $\Aut_k(X)$ is torsion, so that $\Aut_L(X_L)$ is finite (Corollary \ref{cor:exis_of_aut0}).  Therefore, the group $\Aut_L(X_L)$   has no element of infinite order. This proves the corollary.  
 \end{proof}
 
 \begin{remark}  
 Corollary \ref{cor:exis_of_aut} is false in positive characteristic. Indeed, let $E$  be a smooth proper connected genus one curve over $k:=\overline{\mathbb{F}_p}$. Then $E$ has no automorphisms of infinite order over $k$. Let $0$ be an element of $E(k)$, and let $L$ be an uncountable algebraically closed field containing $k$. Then the elliptic curve $(E,0)$ over $k$   has an $L$-point of infinite order, say $x$. Translation by $x$ is an infinite order automorphism of the $L$-scheme $E_L$.  
  \end{remark}

 \begin{remark}\label{remark:les} For all $n\geq 2$, 
 there exists a smooth projective simply connected $n$-dimensional variety $X$ over $\CC$ such that $\mathrm{Aut}^0_{X/\CC}$ is trivial and $\Aut_{\CC}(X)$ is a non-finitely generated (infinite, non-torsion) group; see \cite{Lesieutre}. Note that the arguments  and ideas   in \emph{loc. cit.} are different from those used in our proof of Theorem \ref{thm:tor_is_fin}.
 \end{remark}

 \section{Endomorphisms of arithmetically hyperbolic varieties}\label{section:ends}
   If $f:X\to X$ is an endomorphism of a variety over $k$ and $i\geq 1$, we let $f^i:X\to X$ be the composition of $f$ with itself $i$-times. We let $f^0:= \mathrm{id}_X$. Also, for $f:X\to X$ an endomorphism and $x\in X(k)$, we define $O_f(x) := \{f^k(x)\}_{k\geq 0} $. We will refer to $O_f(x)$ as the  \emph{(forward) $f$-orbit of $x$}.
 
 The crucial ``arithmetic'' observation is the following (very) simple lemma.
%
%
%
 
 \begin{lemma}\label{lem:orbs_are_fin}
 Let $X$ be an arithmetically hyperbolic variety over $k$. Let $f \in \End(X)$  and let $x\in X(k)$. Then the $f$-orbit $O_f(X)$ of $x$ is finite.
 \end{lemma}
 \begin{proof}    Choose a $\ZZ$-finitely generated subring $A\subset k$ and a finite type separated model $\mathcal{X}$ over $A$ such that $f$ descends to an endomorphism  $F:\mathcal{X}\to \mathcal{X}$ of $\mathcal{X}$ over $A$; such data exists by standard spreading out arguments. Moreover,  choose a $\ZZ$-finitely generated subring  $ B\subset k$   containing $A$ such that $x$ lies in the subset $\mathcal{X}(B)$ of $X(k)$.   Now,   the $f$-orbit $O_f(x)\subset X(k)$ of $x$ is contained  in the subset $\mathcal{X}(B)$.   Since   $X$ is arithmetically hyperbolic over $k$, the set $\mathcal{X}(B)$ is finite, so that  $O_f(x)$ is finite.
 \end{proof}

For $X$ a variety over $k$,  recall that a dominant endomorphism $f:X\to X$ has \emph{finite order} if there exist pairwise distinct positive integers $n$ and $m$ such that $f^n = f^m$.  
 
In \cite{Amerik2011}  Amerik proved that dominant endomorphisms    which are not of finite order have points of infinite order. The methods of Amerik are inspired by the work of many authors on dynamical systems of varieties over number fields; see for instance  
  \cite{BellBook}. We will require a mild generalization of Amerik's theorem in which we allow the base field to be an   algebraically closed field of characteristic zero (which is not necessarily  $\Qbar$). This version is well-known to experts to follow from Amerik's line of reasoning. At the request of the referee, we include a brief explanation below.

\begin{theorem} [Amerik]\label{thm:Amerik} Let $X$ be a variety over $k$, and let $f:X\to X$ be a  dominant morphism. If the    orbit of every point $x$ in $X(k)$ is finite, then $f$ has finite order.  
\end{theorem}
\begin{proof}
If $k$ is uncountable, this is ``obvious''. 
If $k=\Qbar$ this is  proven by Amerik \cite[Corollary~9]{Amerik2011}.  The arguments  in Amerik can be used (with minor modifications) to prove the theorem, as we explain now. 

Firstly, by standard ``spreading out'' arguments, we may and do choose the following data.
\begin{enumerate}
\item  A $\ZZ$-finitely generated subring  $A\subset k$;
\item  A finite type separated model $\mathcal{X}$ for   of  $X$ over $A$;
\item A morphism of schemes    $\tilde{f}:\mathcal{X}\to \mathcal{X}$  with $\tilde{f}_k = f$;
\item A prime number $p$, a finite extension $K$ of $\QQ_p $   with ring of integers $\OO_K$, and an embedding $A \subset \OO_K$;
\item A maximal ideal  $\mathfrak{p}\subset \OO_K$;
\item A section  $x\in \mathcal{X}(A)\subset \mathcal{X}(\OO_K)\subset \mathcal{X}(K)$;
\item A dense open affine subscheme  $\mathcal{U}\subset \mathcal{X}$  containing $x$ and a finite surjective morphism of schemes $\mathcal{U}\to \mathbb{A}^n_A$   with $\mathcal{U} = \Spec A[x_1,\ldots,x_n,x_{n+1},\ldots,x_{m}]/I$.
\end{enumerate}
Replacing $A$ by a $\ZZ$-finitely generated, $p$ by a larger prime number if necessary and $K$ by a finite field extension if necessary, we may and do assume that the above data satisfies the following properties.
\begin{enumerate}
\item  every point in the orbit of $x_{\mathfrak{p}}$ is smooth on $\mathcal{X}_{\mathfrak{p}}$ and the orbit of $x_{\mathfrak{p}}$ is  disjoint from the ramification locus of $\tilde{f}_{\mathfrak{p}}$.
\item The coefficients of the power series $x_{n+1},\ldots, x_m, \tilde{f}^\ast x_1,\ldots, \tilde{f}^\ast x_m$ lie in $A$ (when considered as power series in $x_1,\ldots,x_n$).
\item  For all $n+1 \leq i \leq m$, the (monic) minimal polynomial $P_i$ of $x_{i}$ over $K[x_1,\ldots,x_n]$  has coefficients in $A$ and the derivative $P_i'$ of $P_i$ is not identically zero modulo $\mathfrak{p}$.  
\end{enumerate}
 To construct this data,   use   Cassels's embedding theorem to find an embedding of the finitely generated subring $A\subset k$ into a $p$-adic field  (which is arguably the only  ``additional'' ingredient necessary to adapt Amerik's arguments) and, following Amerik, Hrushovski's theorem on intersections of graphs with Frobenius \cite[Corollary~2]{Amerik2011} (which relies on \cite{H, Varshavsky}).     
 
 In the rest of the proof we follow Amerik. Thus, define $\mathcal{N}_{p,x}$ to be  
 \[
 \mathcal{N}_{\mathfrak{p},x} := \{t \in U(K) \ | \ x_i(t) \equiv x_i(x) \ \mathrm{for \ } 1\leq i\leq m\}.
 \]
Now, the $p$-adic ``uniformisation'' theorem of Bell--Ghioca--Tucker as proven by Amerik (see \cite[Proposition~3]{Amerik2011}) implies that the following holds. There is an integer $\ell \geq 1$ and an integer $N\geq 1$ such  that $f^\ell$ maps $\mathcal{N}_{p,x}$ into itself, every preperiodic point in $\mathcal{N}_{p,x}$ has order at most $N$, and the  subset $X(k)\cap \mathcal{N}_{\mathfrak{p},x}$ is dense  in $\mathcal{N}_{p,x}$. This implies the result by \cite[Corollary~8]{Amerik2011}.    
\end{proof}

 \begin{lemma}\label{lem:periodic_is_aut}
 Let $X$ be a variety over $k$. Let $f:X\to X$ be a dominant endomorphism. Suppose that there are distinct positive integers $n$ and $m$     such that $f^n =f^m$. Then $f^{\vert n-m\vert } = \mathrm{id}_X$.   
 \end{lemma}   
    \begin{proof}  The proof is straightforward. Indeed, we may and do assume that $n>m$ and write $g:=f^{n-m}$.
   Let $k\geq 2$ be an integer such that $(k-1)n - km\geq 0$. Define $G:=\left(f^{n-m}\right)^{k-1} = f^{n(k-1)-m(k-1)}$. Let $P\in X$ be in the image of $G$. Let $Q$ be a point such that $G(Q) = P$. Then,
    \begin{eqnarray*}
 f^{n-m}(P) &= & f^{n-m} \circ \left(f^{n-m}\right)^{k-1}(Q) = f^{nk-mk} (Q) = f^{n(k-1)-mk} \circ f^n (Q)  \\
 &=& f^{n(k-1) - mk} (f^m(Q)) = f^{n(k-1)-m(k-1)}(Q) = \left( f^{n-m}\right)^{k-1}(Q) \\ & = & G(Q) =  P.
 \end{eqnarray*} Thus, $P$ is a fixed point of $f^{n-m}$. We now use this observation to show that $f^{n-m}  =\mathrm{id}_X$.
 
   Let $X^g := \{P\in X \ | \ g(P) = P\}$ be the fixed locus of $X$. Since $X$ is separated, $X^g$ is a closed subscheme of $X$. Indeed, $X^g$ is the intersection of $\Delta$ and the graph of $g$ in $X\times X$.
  However, since every point in the image of $G$ is a fixed point of $f^{n-m}$, we see that $X^g$ contains the image of $G$. Moreover, as $f$ is dominant, the morphism $G= \left(f^{n-m}\right)^{k-1}$ is dominant. Thus, the closed subscheme $X^g$ contains the dense subset $G(X)$ of $X$. Therefore, since $X^g$ is closed and dense, it follows that  $X^g =X$. We conclude that $g=\mathrm{id}_X$, as required.  
    \end{proof}
    

We now prove Theorem \ref{thm:ends} and Theorem \ref{thm1}.

 \begin{proof}[Proof of Theorem \ref{thm:ends}] Let $X$ be an arithmetically hyperbolic variety over $k$. We first show that every dominant endomorphism is an automorphism of finite order. Thus,    let $f:X\to X$ be a dominant endomorphism of $X$ over $k$.   Since $X$ is arithmetically hyperbolic, we see that  $f$ has finite orbits (Lemma \ref{lem:orbs_are_fin}).  Thus, it follows from Lemma \ref{lem:periodic_is_aut} that $f$ is an automorphism of finite order. This proves the second statement of the theorem.  Now,  as every automorphism of $X$ has finite order, the group $\Aut_k(X)$ is  a torsion group.  
 Let $\Gamma\subset \Aut_k(X)$ be a finitely generated subgroup. Then, as $\Aut_k(X)$ is torsion, the group $\Gamma$ is a finitely generated torsion subgroup of $\Aut_k(X)$. Therefore, by the theorem of Bass--Lubotzky \cite[Corollary~1.3]{Bass},  the group $\Gamma$ is finite. This shows that $\Aut_k(X)$ is a locally finite group, and  concludes the proof.
 \end{proof}

 \begin{corollary}\label{cor:auts}
 Let $k\subset L$ be an extension of algebraically closed fields of characteristic zero. Let $X$ be an arithmetically hyperbolic projective variety over $k$. Then $\Aut(X_L)$ is finite.
 \end{corollary}
 \begin{proof} Let $X$ be a projective arithmetically hyperbolic variety over $k$. Since $X$ is arithmetically hyperbolic over $k$, it follows that $\Aut_k(X)$ is a torsion group (Theorem \ref{thm:ends}). Thus,  it follows from
Corollary \ref{cor:exis_of_aut0} that $\Aut_L(X_L)$ is finite.
 \end{proof}

  \begin{proof}[Proof of Theorem \ref{thm1}]  This follows from Corollary \ref{cor:auts} (with $k=L$).
 \end{proof}

 \subsection{A remark on hyperk\"ahler varieties}\label{section:ends4}

 A smooth projective variety over $k$ is a \emph{hyperk\"ahler variety over $k$} if $\pi_1^{et}(X)$ is trivial (i.e., $X$ is algebraically simply connected),   and  $\mathrm{H}^{2,0}(X) :=  \mathrm{H}^0(X,\Omega^2_X)$ is a one-dimensional $k$-vector space which can be generated by a non-degenerate form.


 \begin{theorem}\label{thm:hypkah}
 Let $X$ be a  hyperk\"ahler variety over $k$. If $\rho(X)\geq 3$ and  $\Aut_k(X)$ is finite, then  $X$ admits a rational curve over $k$.
 \end{theorem}
\begin{proof}    Without loss of generality, we may and do assume that $k=\CC$. Now, 
 the statement of the theorem  is shown by Kamenova--Verbitsky in the proof of  \cite[Theorem~3.7]{KV}, and relies on earlier work of Boucksom and Huybrechts; see \cite{Boucksom, Huybr}. Indeed, under our assumptions, the K\"ahler  cone does not  coincide with the positive cone, so that the result follows from \cite[Lemma~3.6]{KV}.
\end{proof}

 \begin{theorem}\label{thm:hyp_kah_intro}
 Let $X$ be a smooth projective hyperk\"ahler variety over $k$. If $\rho(X) \geq 3$, then $X$ is not arithmetically hyperbolic over $k$.
 \end{theorem}

 \begin{proof} Let $X$ be a hyperk\"ahler variety  over $k$ with Picard rank $\rho(X)$ at least three. If $\Aut_k(X)$ is infinite, it follows from our main result (Theorem   \ref{thm1}) that $X$ is not arithmetically hyperbolic over $k$. If $\Aut_k(X)$ is finite, then  $X$ has a rational curve   by Theorem \ref{thm:hypkah}, and is therefore not  arithmetically hyperbolic (Lemma \ref{prop:ar_is_gr}).  
 \end{proof}

 \bibliography{refsci}{}
\bibliographystyle{plain}

\end{document}